\documentclass[reqno, a4paper]{amsart}

\usepackage{amssymb}   
\usepackage{thmtools}  

\usepackage{hyperref}
\usepackage[nameinlink, capitalize, noabbrev]{cleveref}

\declaretheorem[style = plain, numberwithin = section]{theorem}

\declaretheorem[style = plain,      sibling = theorem]{lemma}
\declaretheorem[style = plain,      sibling = theorem]{proposition}
\declaretheorem[style = definition, sibling = theorem]{definition}
\declaretheorem[style = definition, sibling = definition]{notation}
\declaretheorem[style = definition, sibling = theorem]{example}
\declaretheorem[style = definition, sibling = theorem]{problem}
\declaretheorem[style = remark,     sibling = theorem]{remark}

\numberwithin{equation}{section}

\newcommand{\N}{\mathbb{N}}   
\newcommand{\Z}{\mathbb{Z}}   
\newcommand{\Q}{\mathbb{Q}}   
\newcommand{\R}{\mathbb{R}}   
\newcommand{\C}{\mathbb{C}}   
\newcommand{\T}{\mathbb{T}}   
\newcommand{\ca}{$\mathrm{C}^*$-algebra}
\DeclareMathOperator{\rc}{rc}
\DeclareMathOperator{\covol}{covol}

\def\subgroup#1{{#1}_\Omega}

\title[Gabor frames over rational lattices]{Criteria for the existence of Schwartz Gabor frames over rational lattices}

\author{Ulrik Enstad}
\address{Department of Mathematics,
University of Oslo,
Moltke Moes vei 35,
0851 Oslo.}
\email{ubenstad@math.uio.no}

\author{Hannes Thiel}
\address{Department of Mathematical Sciences,
Chalmers University of Technology and University of Gothenburg, 
Chalmers Tvärgata 3, 
SE-412 96 Gothenburg.}
\email[]{hannes.thiel@chalmers.se}
\urladdr{www.hannesthiel.org}

\author{Eduard Vilalta}
\address{Department of Mathematical Sciences,
Chalmers University of Technology and University of Gothenburg, 
Chalmers Tvärgata 3, 
SE-412 96 Gothenburg.}
\email[]{vilalta@chalmers.se}
\urladdr{www.eduardvilalta.com}

\thanks{
UE was supported by The Research Council of Norway through project 314048. 
HT and EV were partially supported by the Knut and Alice Wallenberg Foundation (KAW 2021.0140). EV was also partially supported by Spanish Research State Agency (grant No. PID2020-
113047GB-I00/AEI/10.13039/501100011033) and by the Comissionat per Universitats i Recerca de la Generalitat de Catalunya (grant No. 2021-SGR-01015).
}

\begin{document}

\begin{abstract}
We give an explicit criterion for a rational lattice in the time-frequency plane to admit a Gabor frame with window in the Schwartz class. The criterion is an inequality formulated in terms of the lattice covolume, the dimension of the underlying Euclidean space, and the index of an associated subgroup measuring the degree of non-integrality of the lattice. For arbitrary lattices we also give an upper bound on the number of windows in the Schwartz class needed for a multi-window Gabor frame.
\end{abstract}

\maketitle

\section{Introduction}

Given a point $z = (x,\omega) \in \R^d \times \R^d \cong \R^{2d}$, denote by $\pi(z)$ the unitary operator on $L^2(\R^d)$ given by
\begin{equation}
\label{eq:time-frequency-shift}
(\pi(z)f)(t) 
= e^{2\pi i \langle x, t \rangle}f(t-x), 
\qquad \text{ for } f \in L^2(\R^d), \; t \in \R^d. 
\end{equation}

Given a lattice $\Gamma \subseteq \R^{2d}$ and $g \in L^2(\R^d)$, a sequence in $L^2(\R^d)$ of the form
\[ 
\pi(\Gamma)g = (\pi(\gamma)g)_{\gamma \in \Gamma} 
\]
is known as a \emph{Gabor system}, and the function $g$ is called the \emph{window} of the system.
The question of when the Gabor system $\pi(\Gamma)g$ is a frame, that is, when there exist $A,B > 0$ such that
\[ 
A \| f \|_2^2 \leq \sum_{\gamma \in \Gamma}|\langle f, \pi(\gamma)g \rangle|^2 \leq B \| f \|_2^2, 
\qquad \text{ for } f \in L^2(\R^d),
\]
is central to time-frequency analysis and related fields \cite{gr01}.

We will call a Gabor system $\pi(\Gamma)g$ \emph{Schwartz} if the window $g$ belongs to the Schwartz space $\mathcal{S}(\R^d)$ of rapidly decaying smooth functions. 
A version of the Balian--Low Theorem \cite{balian,low,feka04} states that a Schwartz Gabor system can be a frame only if $\covol(\Gamma) < 1$. 
This paper concerns the converse of the Balian--Low Theorem, namely the following:

\begin{problem}
\label{prob:balian-low-converse}
Let $\Gamma$ be a lattice in $\R^{2d}$ with $\covol(\Gamma) < 1$. 
Does there exist a Schwartz function $g \in \mathcal{S}(\R^d)$ such that $\pi(\Gamma)g$ is a frame for $L^2(\R^d)$?
\end{problem}

In dimension $d=1$ \Cref{prob:balian-low-converse} always has an affirmative answer: It is a consequence of the theorems of Lyubarskii, Seip and Wallsten \cite{se92,sewa92,ly92} that every lattice $\Gamma$ with $\covol(\Gamma) < 1$ admits a Schwartz Gabor frame. In fact, $g$ can be chosen to be the Gaussian function $g(t) = e^{-\pi t^2}$.

For every $d \geq 2$ \Cref{prob:balian-low-converse} is still unsolved, see for instance \cite[p.\ 225, Remarks~2]{gr11}. 
However, the problem has a reformulation in terms of comparison of projections in \ca{s} \cite{lu09,lu11,beenva22} that allows operator algebraic methods to be used. This was employed by Jakobsen--Luef \cite[Theorem~5.4]{jalu20}, using results of Rieffel \cite{ri88}, to give an affirmative answer to \Cref{prob:balian-low-converse} in the case of non-rational lattices. 
By definition a lattice $\Gamma \subseteq \R^{2d}$ is \emph{non-rational} if $\Omega(\gamma,\gamma') \notin \Q$ for some $\gamma,\gamma' \in \Gamma$, where $\Omega$ denotes the standard symplectic form on $\R^{2d}$ given by
\begin{equation}
\label{eq:symplectic-form}
\Omega((x,\omega),(x',\omega')) = \langle x, \omega' \rangle - \langle x', \omega \rangle, \qquad \text{ for } (x,\omega),(x',\omega') \in \R^{2d}. 
\end{equation}
Thus, the instance of \Cref{prob:balian-low-converse} that remains open is the case of \emph{rational} lattices, that is, those for which the symplectic forms of all pairs of lattice elements are rational.

The purpose of this paper is to show how comparison theory of projections in \ca{s} can be used to obtain a partial solution of \Cref{prob:balian-low-converse} for rational lattices.
For this, we consider the following invariant of a lattice in $\R^{2d}$:

\begin{definition}
\label{dfn:nGamma}
Let $\Gamma \subseteq \R^{2d}$ be a lattice.
Consider the subgroup of $\Gamma$ given by
\begin{equation}
\label{eq:subgroup}
\subgroup{\Gamma} 
= \big\{ \gamma \in \Gamma : \Omega(\gamma,\gamma') \in \Z \text{ for all } \gamma' \in \Gamma \big\}.
\end{equation}
We define $n_\Gamma = [\Gamma : \subgroup{\Gamma}]^{1/2}$ if the index $[\Gamma : \subgroup{\Gamma}]$ is finite, and $n_\Gamma = \infty$ otherwise.
\end{definition}

We prove in \Cref{prop:rational-lattice} that the index $[\Gamma : \subgroup{\Gamma}]$ is finite if and only if $\Gamma$ is rational.
Moreover, in this case we show in \Cref{prop:rational-torus-homogeneous} that the noncommutative $2d$-torus associated to $\Gamma$ is a $n_\Gamma$-homogeneous \ca{}, i.e. all its irreducible representations have dimension~$n_\Gamma$.
We therefore view $n_\Gamma$ as a measure of the non-integrality of $\Gamma$, that is, the non-integrality of the numbers $\Omega(\gamma,\gamma')$ for $\gamma,\gamma' \in \Gamma$.

Our main result provides a criterion for the existence of a Schwartz Gabor frame:

\begin{theorem}
\label{thm:1}
Let $\Gamma \subseteq \R^{2d}$ be a lattice such that
\begin{equation}
\label{eq:main-inequality}
\covol(\Gamma) < 1 - \frac{d-1}{n_\Gamma}.
\end{equation}
Then there exists $g \in \mathcal{S}(\R^d)$ such that $\pi(\Gamma)g$ is a frame for $L^2(\R^d)$.
\end{theorem}

Thus, the closer the covolume of $\Gamma$ is to being~$1$, the more non-integrality of~$\Gamma$ is required for \eqref{eq:main-inequality} to hold. 
In the extreme case that $\Gamma$ is non-rational, we interpret $(d-1)/n_\Gamma$ as being zero, and then \Cref{thm:1} recovers the known partial solution to \Cref{prob:balian-low-converse} from \cite[Theorem~5.4]{jalu20}. 
Our result also recovers the known solution for $d=1$.

\Cref{thm:1} is a special case of the following result, which gives an upper bound on the number of windows required for a Schwartz multi-window Gabor frame over any lattice $\Gamma \subseteq \R^d$. Due to abstract principles \cite[Theorem~4.5]{lu09} or alternatively a compactness argument \cite[Proposition~6.4]{beenva22} this number is known to be finite, but the proof techniques of \cite{lu09,beenva22} give no control on the number of windows. 
\Cref{thm:2} on the other hand provides a concrete number in terms of $\covol(\Gamma)$, $n_\Gamma$, and~$d$. In the statement below, $\lfloor x \rfloor$ denotes the biggest integer smaller than or equal to $x$.

\begin{theorem}
\label{thm:2}
Let $\Gamma \subseteq \R^{2d}$ be a lattice and set
\[ 
k = \left\lfloor \covol(\Gamma) + \frac{d-1}{n_\Gamma} \right\rfloor + 1. 
\]
Then there exist $g_1, \ldots, g_k \in \mathcal{S}(\R^d)$ such that $(\pi (\gamma) g_j)_{\gamma\in\Gamma , 1\leq j\leq k}$ is a $k$-multiwindow Gabor frame for $L^2(\R^d)$.
\end{theorem}

Note that \Cref{thm:2} also gives a (less optimal) bound independent of $[\Gamma : \subgroup{\Gamma}]$, namely $k = \lfloor \covol(\Gamma) \rfloor + d$.

\section{Rational lattices}\label{sec:rational_lattices}

Denote by $\mathrm{GL}(2d,\R)$ the set of real, invertible $2d \times 2d$ matrices and by $\mathrm{GL}(2d,\Z)$ the subset of integer matrices with determinant equal to $\pm 1$. Let
\[ J = \begin{pmatrix} 0_{d} & I_{d} \\ -I_{d} & 0_{d} \end{pmatrix} \]
denote the standard symplectic $2d \times 2d$ matrix, where $0_d$ denotes a $d \times d$ zero matrix and $I_d$ denotes a $d \times d$ identity matrix. The standard symplectic form defined in \eqref{eq:symplectic-form} can be written as
\[
\Omega(z,w) = \langle z, J w \rangle, 
\qquad \text{ for } z,w \in \R^{2d}.
\]
As already stated in the introduction, a lattice $\Gamma \subseteq \R^{2d}$ is called \emph{rational} if the numbers $\Omega(\gamma,\gamma')$ are rational for all $\gamma,\gamma' \in \Gamma$. 
We can always write a lattice as $\Gamma = M \Z^{2d}$ for some $M \in \mathrm{GL}(2d,\R)$, and we say that $M$ \emph{represents} $\Gamma$ in this case. If $M' \in \mathrm{GL}(2d,\R)$ also represents $\Gamma$, then there exists $R \in \mathrm{GL}(2d,\Z)$ such that $M' = M R$. 
For any $M \in \mathrm{GL}(2d,\R)$ the associated matrix $\theta = M^t J M$ is skew-symmetric, that is, $\theta^t = -\theta$ where $\theta^t$ denotes the transpose of $\theta$. 
If $M' = M R$ then $\theta' := (M')^t J M' = R^t \theta R$, i.e. $\theta$ and $\theta'$ are \emph{congruent}. 
Thus any lattice $\Gamma$ determines a real, skew-symmetric matrix $\theta$ up to congruence, where $\theta = M^t J M$ for some $M \in \mathrm{GL}(2d,\R)$ such that $\Gamma = M \Z^{2d}$. 
In particular $\Gamma$ is rational if and only if $\theta$ is rational, in other words, if and only if $\theta$ has rational entries.

Suppose that $\theta$ is a rational skew-symmetric $2d \times 2d$ matrix. Let $r$ be the least natural number such that $r \theta$ is integral. We will call $r$ the \emph{order} of $\theta$. We will represent $r\theta$ in its \emph{skew normal form} (see \cite[p.\ 57, Theorem IV.1]{Ne72} for an algorithm). For this, there are natural numbers $h_1, \ldots, h_{d}$ with $h_i \mid h_{i+1}$ for $1 \leq i \leq d-1$ and a matrix $R \in GL(\Z,2d)$ such that
\begin{equation}
\label{eq:block_form}
R^t (r \theta) R = \begin{pmatrix} 0 & B \\ -B & 0 \end{pmatrix},
\end{equation}
where
\[
B = \begin{pmatrix} h_1 & \cdots & 0 \\ \vdots & \ddots & \vdots \\ 0 & \cdots & h_d \end{pmatrix}.
\]
The numbers $h_1, \ldots, h_d$ are unique and are called the \emph{invariant factors} of~$r\theta$.

If $\theta$ and $\theta'$ are congruent rational skew-symmetric $2d \times 2d$ matrices, say $\theta' = R^t \theta R$ for some $R \in \mathrm{GL}(2d,\Z)$, then they have the same order. Indeed, $r\theta$ is rational if and only if $r\theta' = R^t (r\theta) R$ is rational. Furthermore, since the matrices $r\theta$ and $r\theta'$ are congruent, they share the same skew normal form, hence they have identical invariant factors as well. For this reason it makes sense to define the \emph{order} and \emph{invariant factors} of a rational lattice $\Gamma \subseteq \R^{2d}$ respectively as the order $r$ of $\theta$ and invariant factors of $r\theta$, where $\theta = M^t J M$ for any $M \in \mathrm{GL}(2d,\R)$ that represents~$\Gamma$.

\begin{proposition}
\label{prop:rational-lattice}
Let $\Gamma \subseteq \R^{2d}$ be a lattice. 
Then $\Gamma$ is rational if and only if the subgroup $\subgroup{\Gamma}$ of $\Gamma$ defined in \eqref{eq:subgroup} has finite index. 
Furthermore, letting $r$ be the order of $\Gamma$ and $h_1, \ldots, h_d$ its invariant factors, we have that
\[ 
\Gamma / \subgroup{\Gamma} \cong ( \Z / r_1 \Z)^2 \times \cdots \times (\Z / r_d \Z)^2 
\]
where $r_i = r / \gcd(h_i,r)$ for $1 \leq i \leq d$. 
In particular $[ \Gamma : \subgroup{\Gamma} ] = r_1^2 \cdots r_d^2$, and thus
\[ 
n_\Gamma
= [ \Gamma : \subgroup{\Gamma} ]^{1/2}
= r_1 \cdots r_d
= \frac{r^d}{\gcd(h_1,r)\cdots\gcd(h_d,r)}.
\]
\end{proposition}

\begin{proof}
First, note that by picking a basis $e_1, \ldots, e_{2d}$ for $\Gamma$, the numbers $\Omega(\gamma,\gamma')$ for $\gamma,\gamma' \in \Gamma$ can all be expressed as integer linear combinations of $\Omega(e_i,e_j)$ for $1 \leq i,j \leq 2d$. Hence $\Gamma$ is rational if and only if there exists $n \in \N$ such that $\Omega(\gamma,\gamma') \in n^{-1}\Z$ for all $\gamma,\gamma' \in \Gamma$. Equivalently $n \gamma \in \subgroup{\Gamma}$ for all $\gamma \in \Gamma$, which means that every element of $\Gamma / \subgroup{\Gamma}$ has order at most $n$. Since $\Gamma / \subgroup{\Gamma}$ is a finitely generated abelian group, this happens exactly when it has finite order, that is, when $[\Gamma : \subgroup{\Gamma} ] < \infty$.

Suppose now that $\Gamma$ is rational. Let $M \in \mathrm{GL}(2d,\R)$ be a matrix that represents~$\Gamma$ and let $\theta = M^t J M$. Since the order and invariant factors are invariant under congruence, we may assume that $r\theta$ is already of the block form in \eqref{eq:block_form}. 
If we identify $\Gamma$ with $\Z^{2d}$ via $\gamma \mapsto M^{-1}\gamma$, then $\Gamma_{\Omega}$ is identified with the subgroup
\[ 
H = \big\{ k \in \Z^{2d} : \Omega(Mk,Ml) \in \Z \text{ for all } l \in \Z^{2d} \big\}.
\]
Note that $k = (k_1, \ldots, k_{2d}) \in H$ if and only if
\[ 
\sum_{i=1}^{2d} k_i \theta_{i,j} = \sum_{i=1}^d (k_i - k_{i+d}) \frac{h_i}{r} \in \Z, 
\qquad \text{ for } 1 \leq j \leq 2d. 
\]
Hence we see that a basis for $H$ is given by the vectors
\[ 
r_1 \cdot e_1, \ldots, r_d \cdot e_d, r_1 \cdot e_{d+1}, \ldots, r_d \cdot e_{2d}, 
\]
where $e_i$ denotes the $i$th standard basis vector of $\R^{2d}$. It follows that $\Z^{2d}/H \cong (\Z / r_1 \Z)^2 \times \cdots \times (\Z / r_d \Z)^2$, which finishes the proof.
\end{proof}

\section{Comparison of projections}

In this section we introduce the relevant notions and results we need from the comparison theory of projections and positive elements in \ca{s}.
These will later be applied to noncommutative tori.

Let $A$ be a \ca{} and let $p$ and $q$ be projections in $A$. 
We say that $p$ is \emph{(Murray-von Neumann) subequivalent} to $q$, written $p \precsim q$, if there exists an element $v \in A$ such that $v^*v = p$ and $vv^* \leq q$. This notion can be extended to matrices with values in $A$ in the following way: Denote by $\mathrm{M}_n(A)$ the set of $n \times n$ matrices with values in $A$ and set
\[ 
\mathrm{M}_\infty(A) = \bigcup_{n=1}^\infty \mathrm{M}_n(A),
\]
where $\mathrm{M}_n(A)$ is identified with the top left corner of $\mathrm{M}_{n+1}(A)$. 
We define subequivalence of two projections $p,q \in \mathrm{M}_\infty(A)$ by viewing them as elements in $\mathrm{M}_n(A)$ for $n$ big enough and applying the definition in this \ca{} (this is independent of the chosen $n$).

For projections $p \in \mathrm{M}_m(A)$ and $q \in \mathrm{M}_n(A)$ we set
\[ 
p \oplus q = \begin{pmatrix} p & 0 \\ 0 & q \end{pmatrix} \in \mathrm{M}_{m+n}(A). 
\]
Thus, for $p,q \in \mathrm{M}_\infty(A)$, we can define $p \oplus q \in \mathrm{M}_\infty(A)$. 
For $n \in \N$ we also set $p^{\oplus n} = p \oplus \cdots \oplus p$ where the number of summands equals $n$.

A \emph{tracial state} on a \ca{} $A$ is a positive linear functional $\tau \colon A \to \C$ such that $\| \tau \| = 1$ and $\tau(ab) = \tau(ba)$ for all $a,b \in A$.
By abuse of notation we denote also by $\tau$ the function on $\mathrm{M}_\infty(A)$ given by $a \mapsto \sum_{i=1}^n \tau(a_{i,i})$ for $a = (a_{i,j})_{i,j=1}^n \in \mathrm{M}_n(A)$.

One of the central tools of the present paper is the \emph{radius of comparison} of a \ca{} $A$, denoted by $\rc(A)$, which is a number in $[0,\infty]$; 
see \cite[Definition~6.1]{Tom06FlatDimGrowth} for (residually) stably finite \ca{s}, and \cite{BlaRobTikTomWin12AlgRC} for an amended definition in the general setting.
It encodes the necessary size of the gap between the ranks of two positive elements in some matrix algebra over $A$ to ensure their subequivalence in the sense of Cuntz.
Since the precise definition is not relevant here, we only state a specific application to the comparison of projections.

\begin{lemma}
\label{lem:radius}
Let $A$ be a unital, exact \ca{}, and let $p,q$ be projections in $\mathrm{M}_\infty(A)$.
Suppose that
\[
\tau(p)  < \tau(q) - \rc(A)
\]
for every tracial state $\tau$ on $A$. 
Then $p \precsim q$.
\end{lemma}
\begin{proof}
This follows from the definition in \cite[Section~3.1]{BlaRobTikTomWin12AlgRC}, using that for exact \ca{s} all quasitraces $\tau$ are traces, that for a projection $r$ we have $d_\tau([r])=\tau(r)$, and that Murray-von Neumann subequivalence and Cuntz subequivalence agree on projections.
\end{proof}

A \ca{} is said to be \emph{$n$-homogeneous} if each of its irreducible representations is $n$-dimensional.
To simplify the discussion, we restrict to the unital case.
Here, the typical example is $M_n(C(X))$ for some compact, Hausdorff space $X$, which we can view as the algebra of continuous sections of the trivial bundle over~$X$ with fibers $M_n(\C$).
More generally, given a locally trivial $M_n(\C)$-bundle over a compact, Hausdorff space $X$, the algebra of continuous sections is a unital, $n$-homogeneous \ca{} with primitive ideal space naturally homeomorphic to $X$. All unital $n$-homogeneous \ca{s} arise this way;
see for example \cite[Theorem~IV.1.7.23]{Bla06OpAlgs}.

We will need the following estimate of the radius of comparison of $n$-homogeneous \ca{s}, which is a consequence of \cite[Theorem~4.6]{Tom09CompThe}.

\begin{proposition}[Toms]
\label{prop:radius-homogeneous}
Let $A$ be a unital $n$-homogeneous \ca{} with primitive ideal space $X$. 
Then
\[ 
\rc(A) \leq \frac{\dim(X)-2}{2n}. 
\]
\end{proposition}

\section{Noncommutative tori}

Let $\theta$ be a real, skew-symmetric $2d \times 2d$ matrix, that is, $\theta^t = -\theta$ where $\theta^t$ denotes the transpose of $\theta$. 
The \emph{noncommutative torus} $A_{\theta}$ associated to $\theta$ is the universal \ca{} generated by $2d$ unitary elements $u_1, \ldots, u_{2d}$ satisfying
\begin{equation}
\label{eq:nc-torus-relations}
u_j u_i = e^{2\pi i \theta_{i,j}} u_i u_j, 
\qquad \text{ for } 1 \leq i,j \leq 2d.
\end{equation}
As in \Cref{sec:rational_lattices}, two real skew-symmetric matrices $\theta$ and $\theta'$ are congruent if there exists an $R \in \mathrm{GL}(2d,\Z)$ such that $\theta' = R^t \theta R$. In this case the associated noncommutative tori $A_\theta$ and $A_{\theta'}$ are isomorphic \cite[p.\ 293]{RiSc99}.

The noncommutative torus $A_\theta$ is equipped with a canonical tracial state $\tau$ which is determined by
\[ 
\tau( u_1^{k_1} \cdots u_{2d}^{k_{2d}}) = \delta_{k_1,0} \cdots \delta_{k_{2d},0} , \quad k_1, \ldots, k_{2d} \in \Z. 
\]
For a proof of the following proposition see \cite[Lemma~2.3]{el84}.

\begin{proposition}
\label{prop:traces-on-torus}
Let $p$ be a projection in $\mathrm{M}_\infty(A_\theta)$ and let $\tau'$ be any tracial state on $A_\theta$. Then $\tau'(p) = \tau(p)$ where $\tau$ denotes the canonical tracial state.
\end{proposition}

We call a noncommutative torus $A_\theta$ \emph{rational} if the associated skew-symmetric matrix $\theta$ only has rational entries. 
It is known that rational noncommutative tori are homogeneous \ca{s} with spectrum homeomorphic to the torus $\T^{2d}$.
The dimension of the irreducible representations of a rational noncommutative torus is computed in \cite{vava86}, but we include the computation here for completeness.

\begin{proposition}
\label{prop:rational-torus-homogeneous}
Let $\theta$ be a rational skew-symmetric nondegenerate $2d \times 2d$ matrix of order $r \in \N$ and let $h_1, \ldots, h_d$ be the invariant factors of $r \theta$. 
Set
\begin{equation}
\label{eq:homogeneous-dimension}
n = \frac{r^d}{\gcd(h_1,r) \cdots \gcd(h_d,r)}.
\end{equation}
Then $A_\theta$ is $n$-homogeneous with spectrum homeomorphic to $\T^{2d}$.
\end{proposition}

\begin{proof}
According to the skew normal form, we can find $R \in \mathrm{Gl}(2d,\Z)$ such that \eqref{eq:block_form} holds, with corresponding diagonal matrix $B$ having the invariant factors on the diagonal. 
It follows that
\[
R^t\theta R = \begin{pmatrix} 0 & r^{-1}B \\ -r^{-1}B & 0 \end{pmatrix}. 
\]
Hence $A_{\theta} \cong A_{R^t \theta R}$ is universally generated by unitaries $u_1, \ldots, u_{2d}$ where $u_{j+d} u_j = e^{2\pi i h_j/r} u_j u_{j+d}$ for $1 \leq j \leq d$ and where all the other unitaries commute. 
In other words, $A_{\theta}$ is isomorphic to the tensor product
\[ 
A_{\theta} \cong A_{h_1/r} \otimes \cdots \otimes A_{h_{d}/r} 
\]
of~$d$ noncommutative $2$-tori.

It is well-known that the noncommutative $2$-torus $A_{p/q}$ where $p,q \in \Z$, $q > 0$ and $\gcd(p,q) = 1$ is $q$-homogeneous. 
An explicit irreducible $q$-dimensional representation $\pi_{p,q}$ is given by $u \mapsto U$ and $v \mapsto V$, where the $q \times q$ matrices $U$ and $V$ are given by
\[ 
U = \begin{pmatrix} 0 & 1 & 0 & \cdots & 0 \\ 0 & 0 & 1 & \cdots & 0 \\ \vdots & \vdots  & \ddots & \ddots & \vdots \\ 0 & 0 & 0& \ddots & 1 \\ 1 & 0 & 0 & \cdots & 0 \end{pmatrix}, \qquad V = \begin{pmatrix} 1 & 0 & \cdots & 0 \\ 0 & e^{2\pi i p/q} & \cdots & 0 \\ \vdots & \vdots & \ddots & \vdots \\ 0 & 0 & \cdots & e^{2\pi i p (q-1)/q} \end{pmatrix}. 
\]
We form the numbers
\[ 
p_i = \frac{h_i}{\gcd(h_i,r)}, \;\;\; q_i = \frac{r}{\gcd(h_i,r)}, \quad 1 \leq i \leq s, 
\]
so that $\gcd(p_i,q_i) = 1$. 
The \ca{} $A_{h_i/r} = A_{p_i/q_i}$ is then $q_i$-homogeneous with spectrum $\T^2$.

In general, if~$A$ is a unital, $n$-homogeneous \ca{} with spectrum $X$ and~$B$ is unital, $k$-homogeneous with spectrum $Y$, then $A \otimes B$ is $nk$-homogeneous with spectrum $X \times Y$.
Indeed, if $\pi_A$ and $\pi_B$ are irreducible representations of $A$ and~$B$, respectively, then $\pi_A\otimes\pi_B$ is an irreducible representation of $A \otimes B$, and by \cite[IV.3.4.25]{Bla06OpAlgs} every irreducible representation of $A \otimes B$ arises this way up to unitary equivalence. 
(This works as soon as one of the \ca{s} is type~I, which includes homogeneous \ca{s}.)

Using that $A_{\theta} \cong A_{h_1/r} \otimes \cdots \otimes A_{h_{d}/r}$, we obtain that $A_\theta$ is $q_1 \cdots q_d$-homogeneous with spectrum homeomorphic to $\T^{2d} = \T^2 \times \cdots \times \T^2$, and the result follows since
\[ 
n = q_1 \cdots q_d = \frac{r^d}{\gcd(h_1,r) \cdots \gcd(h_d,r)}. \qedhere
\]
\end{proof}

Using \Cref{prop:rational-torus-homogeneous} we can now infer the following about the radius of comparison of noncommutative tori.

\begin{proposition}
\label{prop:radius-nc-tori}
Let $\theta$ be a real, skew-symmetric $2d \times 2d$ matrix.
\begin{enumerate}
\item 
If $\theta$ is non-rational, then $A_\theta$ has strict comparison, that is, $\rc(A_\theta) = 0$.
\item 
If $\theta$ is rational, then
\[ 
\rc(A_\theta) \leq \frac{d-1}{n} 
\]
where $n$ is as in \Cref{prop:rational-torus-homogeneous}.
\end{enumerate}
\end{proposition}

\begin{proof}
(1): We know from \cite[Theorem~1.5]{BlaKumRor92} that $A_{\theta}$ is approximately divisible. Then, a combination of \cite[Lemma~3.8]{BlaKumRor92} and \cite[Proposition~6.2]{EllRobSan11} shows that $A_\theta$ has strict comparison.

(2):
By \Cref{prop:rational-torus-homogeneous}, $A_\theta$ is $n$-homogeneous with spectrum $\T^{2d}$.
Applying \Cref{prop:radius-homogeneous}, we get
\[
\rc(A_\theta)
\leq \frac{\dim(\T^{2d})-2}{2n}
= \frac{d-1}{n},
\]
as desired.
\end{proof}

\section{Proofs of main results}

The time-frequency shifts defined in \eqref{eq:time-frequency-shift} satisfy the commutation relation
\begin{equation}
\label{eq:commutation-relation}
\pi(w)\pi(z) = e^{2\pi i \Omega(z,w)} \pi(z)\pi(w), 
\qquad \text{ for } z,w \in \R^{2d}. 
\end{equation}

Let $\Gamma$ be a lattice in $\R^{2d}$. We consider the \ca{} $C_\pi^*(\Gamma)$ generated by $\pi(\Gamma) = \{ \pi(\gamma) : \gamma \in \Gamma \}$, that is, the smallest operator norm-closed, self-adjoint algebra of operators on $L^2(\R^d)$ that contains $\pi(\Gamma)$. 
Say that the lattice is represented by a matrix $M \in GL(2d,\R)$ and let $\theta = M^t J M$ be the associated real skew-symmetric $2d \times 2d$ matrix. 
The entries of $\theta$ are given by $\theta_{i,j} = \Omega(Me_i,Me_j)$ where $e_i$ denotes the $i$th standard basis vector of $\R^{2d}$. 
Hence, according to \eqref{eq:commutation-relation}, the unitary operators $u_i = \pi(Me_i)$, $1 \leq i \leq 2d$, generating $C_\pi^*(\Gamma)$ satisfy the commutation relations \eqref{eq:nc-torus-relations}. 
They thus define a representation of $A_{\theta}$ on $L^2(\R^d)$ which can be shown to be faithful (see e.g.\ \cite[Proposition~2.2]{ri88}).
 Hence the \ca{} $C^*(\pi(\Gamma))$ is isomorphic to $A_\theta$ for $\theta = M^t J M$.

Note that if $\Gamma$ is represented by another matrix $M' \in \mathrm{GL}(2d,\R)$, then $\theta' = (M')^t J M$ is congruent to $\theta$, so the associated noncommutative torus $A_{\theta'}$ is isomorphic to $A_\theta$. Hence the lattice $\Gamma$ uniquely determines a noncommutative torus. Furthermore, $\Gamma$ is rational if and only if $A_\theta$ is rational.

In \cite{ri88}, Rieffel shows that the Schwartz space $\mathcal{S}(\R^d)$ can be completed into a finitely generated projective module $\mathcal{E}_{\Gamma}$ over $C_\pi^*(\Gamma)$ where the action is implemented by time-frequency shifts. Since the module is finitely generated and projective, it is represented by a projection in a matrix algebra over $C_\pi^*(\Gamma)$. We make the following notation for this projection.

\begin{notation}
\label{not:projection}
We denote by $p_\Gamma \in \mathrm{M}_\infty(C_\pi^*(\Gamma))$ the projection representing the Heisenberg module $\mathcal{E}_\Gamma = \overline{\mathcal{S}(\R^d)}$ over $C_\pi^*(\Gamma)$. That is, if $p_\Gamma \in \mathrm{M}_n(C_\pi^*(\Gamma))$ for some $n \in \N$, then $\mathcal{E}_\Gamma \cong C_\pi^*(\Gamma)^n p_\Gamma$.
\end{notation}

By \cite[Theorem~3.4]{ri88} we have that
\begin{equation}
\label{eq:covolume-projection}
\tau(p_\Gamma) = \covol(\Gamma) 
\end{equation}
where $\tau$ denotes the canonical tracial state on $C_\pi^*(\Gamma)$.

In \cite{lu09,lu11} (see also \cite[Theorem A (iii)]{AuEn20}) Luef established the link between Gabor frames with windows in the Schwartz class and generators of these modules.

\begin{proposition}
\label{prop:gabor_frame_projection}
The following are equivalent:
\begin{enumerate}
\item 
$\Gamma$ admits a $k$-multiwindow Gabor frame with windows in the Schwartz space.
\item 
$p_\Gamma \precsim 1_{A_\theta}^{\oplus k}$.
\end{enumerate}
\end{proposition}

We are now in position to prove the theorems from the introduction.

\begin{proof}[Proof of \Cref{thm:2}]
Set $s = (d-1)/n_\Gamma$ if $\Gamma$ is rational and $s = 0$ otherwise, so that by \Cref{prop:radius-nc-tori}, we have that $\rc(C_\pi^*(\Gamma)) \leq s$ in both cases.

Set $k = \lfloor \covol(\Gamma) + s \rfloor + 1$.
Then
\[
\covol(\Gamma) 
< k-s
\leq k - \rc(C_\pi^*(\Gamma)).
\]

By \eqref{eq:covolume-projection} and \Cref{prop:traces-on-torus}, we have
\[
\tau'(p_\Gamma) = \covol(\Gamma)
< k - \rc(C_\pi^*(\Gamma))
= \tau'(1^{\oplus k}) - \rc(C_\pi^*(\Gamma))
\]
for every tracial state $\tau'$ on $C_\pi^*(\Gamma)$. 
We conclude from \Cref{lem:radius} that $p_\Gamma \precsim 1^{\oplus k}$. 
By \Cref{prop:gabor_frame_projection} we infer that $\Gamma$ admits a $k$-multiwindow Gabor frame with windows in the Schwartz class.
\end{proof}

\begin{proof}[Proof of \Cref{thm:1}]
The assumption that $\covol(\Gamma) < 1 - (d-1)/n_\Gamma$ implies that $k = 1$, where $k$ is as in \Cref{thm:2}. Hence the result follows.
\end{proof}

\section{Applications}

In this section we provide some examples where \Cref{thm:1} can applied. First we reformulate \Cref{thm:1} in terms of the order and invariant factors of a rational lattice.

\begin{proposition}
\label{prop:condition-invariant-factors}
Let $\Gamma \subseteq \R^{2d}$ be a rational lattice with order $r$ and invariant factors $h_1, \ldots, h_d$. 
If
\begin{equation}
    h_1 \cdots h_d < r^d - (d-1) \gcd(r,h_1) \cdots \gcd(r,h_d), \label{eq:inequality-invariant-factors}
\end{equation}
then $\Gamma$ admits a Schwartz Gabor frame.
\end{proposition}

\begin{proof}
Let $\Gamma = M \Z^{2d}$, $\theta = M^t J M$ and let $R \in \mathrm{GL}(2d,\Z)$ be such that $R^t(r\theta)R$ is in its skew normal form. Then
\begin{align*}
\covol(\Gamma) 
&= \det M = (\det \theta)^{1/2} = ( \det R^t \theta R)^{1/2} \\
&= \Big( \det \begin{pmatrix} 0 & r^{-1}B \\ -r^{-1}B & 0 \end{pmatrix} \Big)^{1/2} = \frac{h_1 \cdots h_d}{r^d}. 
\end{align*}
Hence the inequality in \Cref{thm:1} translates into \eqref{eq:inequality-invariant-factors} when combined with \Cref{prop:rational-lattice}.
\end{proof}

\begin{example}
As defined in e.g.\ \cite[Definition~2]{GroHanHeil02}, a lattice $\Gamma\subseteq\mathbb{R}^{2d}$ is \emph{symplectic} if $\Gamma = \alpha M \mathbb{Z}^{2d}$ where $\alpha\in\mathbb{R}\setminus\{ 0\}$ and $M$ is a \emph{symplectic matrix}, that is, $M^t J M=J$. The frame property of $\pi(\Gamma)g$ is equivalent to the frame property of $\pi(\alpha \Z^{2d})(Tg)$ where $T$ is an associated invertible operator, see \cite[Section 9.4]{gr01}. Since a Gabor frame can constructed over $\alpha \Z^{2d}$ using one-dimensional Gaussians whenever $|\alpha| < 1$, it follows that symplectic lattices admit Gabor frames with window in the Schwartz class whenever $| \alpha | < 1$. This can be alternatively shown using \Cref{prop:condition-invariant-factors}. Indeed, note that a rational symplectic lattice of covolume less than $1$ is represented by a matrix $\alpha M$ with $M$ symplectic and $\alpha^2 = \tfrac{p}{q}$ with $p<q$ coprime. It follows that the associated skew-symmetric matrix is $\theta = (\alpha M^t)J(\alpha M)=\alpha^2 J$. In particular, the order of $\theta$ is $r=q^2$. The invariant factors $h_1,\ldots ,h_d$ are all equal to $p^2$.

Now, since $p\leq q-1$, one gets $0<q^{2d}-p^{2d}-(d-1)$ for every $d\geq 1$. Thus, we have
\[
    h_1 \cdots h_d = p^{2d} < q^{2d}-(d-1) = r^d - (d-1) \gcd(r,h_1) \cdots \gcd(r,h_d).
\]

It follows from \Cref{prop:condition-invariant-factors} that $\Gamma$ admits a Schwartz Gabor frame whenever $\covol(\Gamma) < 1$.
\end{example}

\begin{example}
Let $M$ be the $2d \times 2d$ matrix
\[ \begin{pmatrix} 1 & 0 & \cdots & 0 \\ 0 & 1 & \cdots & 0 \\ \vdots & \vdots & \ddots & 0 \\ 0 & 0 & \cdots & 1/q \end{pmatrix} \]
where $q \in \N$. Then
\[ \theta = M^t J M = \begin{pmatrix} 0 & B \\ -B & 0 \end{pmatrix} \]
where $B$ is the $d \times d$ matrix with diagonal $1, \ldots, 1, 1/q$. Thus the order of $\theta$ is $r = q$ while the invariant factors of $r\theta$ are $h_1 = 1$, $h_i = q$ for $2 \leq i \leq d$. Thus by \Cref{prop:condition-invariant-factors} $\Gamma$ admits a Schwartz Gabor frame whenever
\[ q^{d-1} < q^d - (d-1)q^{d-1} , \]
that is, $q > d$. Obviously this fails to hold if $d$ is large relative to $q$.
\end{example}

\begin{example}
Any lattice $\Gamma$ in $\R^{2d}$ with $\covol(\Gamma) < 1/d$ admits a Schwartz Gabor frame. Indeed, if $\covol(\Gamma) < 1 - (d-1)/n_\Gamma$ then the result follows from \Cref{thm:1}. Otherwise, assume that $\covol(\Gamma) \geq 1 - (d-1)/n_\Gamma$, which implies $n_\Gamma<d$. Further, we know that
 \[
 \covol (\Gamma )=\frac{h_1 \cdots h_d}{r^d}=
 \frac{1}{n_\Gamma}\left(\frac{h_1\cdots h_d}{\gcd (h_1, r)\cdots \gcd (h_d , r)}\right)
 \]
 where note that the second factor is an integer. Thus, $\covol (\Gamma )\geq 1/n_\Gamma$. By our previous computations, we obtain $\covol (\Gamma ) \geq 1/d$. This contradicts our assumption.

This result should be compared with \cite[Theorem 1.1]{LuWa23}, which states that almost all lattices with $\covol(\Gamma) < d!/d^d$ admit a Gabor frame with a Gaussian window.
\end{example}

We are currently not aware of a Schwartz Gabor frame over a lattice that violates the inequality in \Cref{thm:1}.

\begin{remark}
We conclude by remarking that the abstract techniques in this paper cannot be used to completely settle \Cref{prob:balian-low-converse}. Indeed, rational noncommutative tori do not have strict comparison, that is, they have strictly positive radius of comparison. For example, it follows from \cite[Corollary~1.2]{EllNiu13RadiusCompCommutative} (see also \cite[Section~4.1]{BlaRobTikTomWin12AlgRC}) that 
\[
\rc(C(\T^{2n+1})) = n-1 ,\quad\text{and}\quad
\rc(C(\T^{2n})) \in \{ n-2, n-1 \}
\]
for each $n\in\N$.

More concretely, one of the main reasons why these techniques do not fully solve \Cref{prob:balian-low-converse} is that, while for this problem it is enough to show that $\tau (p_\Gamma) <\tau (1)$ implies $p_\Gamma\precsim 1$ (\Cref{prop:gabor_frame_projection}), the radius of comparison studies tracial comparison (in the sense of \Cref{lem:radius}) of all pairs of positive elements. Thus, a more detailed analysis of the canonical trace and the distinguished projection $p_\Gamma$ is needed.
\end{remark}

\bibliography{bibl}
\bibliographystyle{abbrv}

\end{document}